\documentclass{article}

\usepackage{xcolor}

\usepackage{amsthm,cite}
\usepackage{amsmath}
\usepackage{amssymb}

\usepackage{enumerate}

\DeclareMathOperator{\Ric}{Ric}
\DeclareMathOperator{\Hess}{Hess}
\DeclareMathOperator{\trace}{trace}
\let\oldref\ref
\renewcommand{\ref}[1]{(\oldref{#1})}

\newtheorem{theorem}{Theorem}[section]
\newtheorem{lemma}[theorem]{Lemma}
\newtheorem{proposition}[theorem]{proposition}
\theoremstyle{definition}
\newtheorem{example}{Example}
\newtheorem{corollary}{Corollary}
\newtheorem{definition}{Definition}

\title{$\eta-$Ricci solitons on contact pseudo-metric manifolds}

\author{Eftekhar Asgharzadeh\thanks{Email:ef.asgharzadeh@gmail.com} and Morteza Faghfouri\thanks{Email:faghfouri@tabrizu.ac.ir}\thanks{Department of Pure Mathematics, Faculty of Mathematical Sciences, University of Tabriz, Tabriz, Iran.}}
\begin{document}
\maketitle

\begin{abstract}
  In this paper, we prove that a Sasakian pseudo-metric manifold which admits an $\eta-$Ricci soliton is an $\eta-$Einstein manifold, and if the potential vector field of the $\eta-$Ricci soliton is not a Killing vector field then the manifold is $\mathcal{D}-$homothetically fixed, and the vector field leaves the structure tensor field invariant. Next, we prove that  a $K-$contact pseudo-metric manifold with a gradient $\eta-$Ricci soliton metric is $\eta-$Einstein. Moreover, we study  contact pseudo-metric manifolds  admitting  an $\eta-$Ricci soliton with a potential vector field point wise colinear with the Reeb vector field. Finally, we study gradient $\eta-$Ricci solitons on $(\kappa, \mu)$-contact pseudo-metric manifolds.
\end{abstract}

\section{Introduction}
A Ricci soliton is a natural generalization of an Einstein metric, which was introduced by Hamilton \cite{hamilton1988ricci} as the fixed point of the Hamilton's Ricci flow $\frac{\partial}{\partial t} g=-2\Ric$. The Ricci flow is a nonlinear diffusion equation analogue of the heat equation  for metrics. A Ricci soliton $(g, V, \lambda)$ on the pseudo-Riemannian manifold $(M, g)$ is defined by the following equation
\begin{equation*}
  \mathsterling_V g+2\Ric+2\lambda g=0,
\end{equation*}
where $\mathsterling_V$ is the Lie derivative along the potential vector field $V$, and $\lambda$ is a constant real number. The Ricci soliton is called shrinking, steady, expanding if $\lambda<0$, $\lambda =0$ and $\lambda >0$, respectively. If $V=Df$, where $Df$ is the gradient of the smooth function $f$, then the Ricci soliton is called a gradient Ricci soliton. Ricci solitons have been studied in many different contexts (see \cite{chodosh2016rotational, bagewadi2012ricci, cualin2010eisenhart, sharma2008certain, futaki2009transverse, nagaraja2012ricci, bejan2014second, ghosh2014sasakian}). Also, they are  interests of physicists because of their relations to string theory \cite{akbar2009ricci, kholodenko2008towards}, and physicists refer to Ricci solitons as quasi-Einstein metrics \cite{friedan1980nonlinear}.

The $\eta-$Ricci soliton notion, as a generalization of a Ricci soliton, was introduced by Cho and Kimura \cite{cho2009ricci}. An $\eta-$Ricci soliton on a manifold $M$ is a tuple $(g, V, \lambda, \mu)$, where $g$ is a pseudo-Riemannian metric, $V$ is the potential vector field, and $\lambda, \mu$ are constant real numbers satisfying
\begin{equation}
     \mathsterling_V g+ 2 \Ric+2\lambda g+2 \mu \eta\otimes\eta=0, \label{pre:etaRicciSolitonEq}
\end{equation}
where $\eta$ is a $1-$form on $M$. Moreover, if  $V=Df$, the $\eta-$Ricci soliton is called a gradient $\eta-$Ricci soliton and Eq.\ref{pre:etaRicciSolitonEq} becomes
\begin{equation}
  \Hess f+\Ric+\lambda g+\mu\eta\otimes\eta=0.\label{pre:gradientEtaRicciSolitonEq}
\end{equation}
The $\eta-$Ricci solitons have been studied in many different settings, Blaga  studied $\eta-$Ricci solitons on para-Kenmotsu  \cite{blaga2015eta} and  Lorentzian para-Sasakian manifolds \cite{blaga2016}. Devaraja and Venkatesha studied $\eta-$Ricci solitons on para-Sasakian manifolds \cite{naik2019eta}, etc.

Contact geometry is an odd-dimensional analogue of the symplectic geometry and has been studied in many different contexts (particularly) those related to physics. It has been used as a proper framework for classical thermodynamics \cite{bravetti2015contact, van2018geometry}, and as a geometrical approach to magnetic field \cite{cabrerizo2009contact}. Also, it was studied in relation with the Yang-Mills theory \cite{kallen2012twisted}, quantum mechanics \cite{herczeg2018contact}, gravitational waves \cite{low1998stable}, etc. Studying contact structures with pseudo-Riemannian metrics was started by Takahashi in \cite{takahashi1969sasakian}, but he just studied the Sasakian case. Recently,  Calvaruso and Perrone \cite{calvaruso2010contact} have studied a contact pseudo-metric manifold in the general case.  Ghaffarzadeh and second author studied nullity conditions on the contact pseudo-metric manifolds  and have introduced the  ``$(\kappa, \mu)-$contact pseudo-metric manifold" notion \cite{ghaffarzadeh}. The relevance for the  general relativity of contact pseudo-metric manifolds was studied in \cite{duggal1990space}. All of these applications have motivated us to study $\eta-$Ricci solitons in the contact pseudo-Riemannian settings.

The present paper has been organized as follows. In Section 2, we recalled the contact pseudo-metric manifold notion and proved some lemmas that are  used in the next sections. In Section 3, we studied $\eta-$Ricci solitons on Sasakian pseudo-metric manifolds and showed that a Sasakian pseudo-metric manifold, which admits an $\eta-$Ricci soliton, is an $\eta-$Einstein manifold and if the potential vector field of the $\eta-$Ricci soliton is not  a Killing vector field, then the manifold is $\mathcal{D}-$homothetically fixed, and presented an  example for it. Moreover, we showed  a $K-$contact pseudo-metric manifold  which admits a gradient $\eta-$Ricci soliton is an $\eta-$Einstein manifold. Also,  we studied  an $\eta-$Ricci soliton that has a  potential vector field colinear to  the Reeb vector field on a contact pseudo-metric manifold and  showed that the manifold is  $K-$contact. In the last section, we studied gradient $\eta-$Ricci solitons on a $(\kappa, \mu)-$contact pseudo-metric manifold and obtained some conditions on the curvature tensor of the manifold.

\section{Preliminaries}\label{preliminaries}
In this section, we recall some definitions and results needed in the rest of the paper.

A $(2n+1)-$dimensional manifold $M$ is called an almost contact pseudo-metric manifold, if there exists an almost  contact pseudo-metric structure $(\varphi, \xi, \eta, g)$ on $M$, where $\varphi, \xi, \eta,g$ are a $(1, 1)-$tensor field, a vector field, a $1$-form and a compatible pseudo-Riemannian metric, respectively, which satisfy the following equations
\begin{gather}
  \eta(\xi) = 1,\quad\varphi^2(X)=-X+\eta(X)\xi,\label{pre:basicPro4}\\
  g(\varphi X, \varphi Y)=g(X, Y)-\epsilon\eta(X)\eta(Y),
\end{gather}
where $\epsilon=\pm 1$, and $X, Y$ are arbitrary vector fields. Using the above equations, we have
\begin{gather*}
  \varphi\xi=0,\quad\eta\circ\varphi=0,\\
  \eta(X) = \epsilon g(\xi, X),\quad g(\varphi X, Y)=-g(X,\varphi Y)
\end{gather*}
and especially $g(\xi, \xi)=\epsilon$. Notice that the signature of the metric $g$ is $(2p+1, 2n-2p)$ if $\xi$ is a spacelike vector field $(g(\xi, \xi)>0)$, and is $(2p, 2n-2p+1)$ if $\xi$ is a timelike vector field ($g(\xi, \xi)<0$).

The fundamental $2-$form $\Phi$ of an almost contact pseudo-metric manifold $(M, \varphi, \xi, \eta, g)$ is defined as $\Phi(X, Y) = g(X, \varphi Y)$, where $X, Y\in\Gamma(M)$. If
\begin{equation}
  g(X, \varphi Y) = (d\eta)(X, Y),\label{pre:basicPro1}
\end{equation}
then $\eta$ is a contact form, $(\varphi, \xi, \eta, g)$ is a contact pseudo-metric structure and $M$ is called a contact pseudo-metric manifold.

Throughout this paper, we use $R(X, Y)=[\nabla_X, \nabla_Y]-\nabla_{[X, Y]}$, where $X, Y\in\Gamma(M),$ as the Riemannian curvature tensor definition. In a contact pseudo-metric manifold $(M, \varphi, \xi, \eta, g)$ the $(1, 1)-$tensor field $\ell$ and $h$ are defined by
\begin{equation*}
  \ell X=R(X, \xi)\xi,\quad hX=\frac{1}{2}(\mathsterling_\xi\varphi)X.
\end{equation*}
Also, notice the $\ell$ and $h$ are self-adjoint operators. In the contact pseudo-metric manifold $(M, \varphi, \xi, \eta, g)$, we have the following equations \cite{calvaruso2010contact,perrone2014curvature}
\begin{gather}
  \trace(h)=\trace(h\varphi)=0,\\
  \eta\circ h=0, \quad \ell\xi=0, \\
  h\varphi=-\varphi h,\quad h\xi=0,\\
  \nabla_\xi\varphi=0,\quad\\
  \nabla_X\xi=-\epsilon\varphi X-\varphi hX,\label{pre:basicPro3}\\
  \Ric(\xi, \xi)=2n-trh^2,\label{pre:basicPro2}
\end{gather}
where $X$ is an arbitrary vector field.

A contact pseudo-metric manifold $(M, \varphi, \xi, \eta, g)$ is a $K-$contact pseudo-metric manifold if $\xi$ is a Killing vector field or equivalently $h=0$. So, we have the following equations
\begin{gather}
  Q\xi=2n\epsilon\xi,\label{revBasic:pro1}\\
  \nabla_X\xi = -\epsilon\varphi X,\label{pre:kContactFormula1}
\end{gather}
where $Q$ is the Ricci operator of the metric $g$ and $X\in\Gamma(M).$
\begin{lemma}\label{pre:kContactCovDiffLemma}
  Let $(M, \varphi, \xi, \eta, g)$ be a $(2n+1)-$dimensional $K-$contact pseudo-metric manifold, then
  \begin{gather}
    (\nabla_XQ)\xi=-2n\varphi X+\epsilon Q\varphi X,\label{pre:kContactCovDiffX}\\
    (\nabla_\xi Q)X=\epsilon(Q\varphi-\varphi Q)X,\label{pre:kContactCovDiffXi}
  \end{gather}
  where $X$ is an arbitrary vector field.
\end{lemma}
\begin{proof}
  First, differentiating \ref{revBasic:pro1} along an arbitrary vector field $X$ and using \ref{pre:kContactFormula1}, we obtain \ref{pre:kContactCovDiffX}. Next Lie differentiating $\Ric$, along $\xi$, we find
  \begin{equation*}
    (\mathsterling_\xi \Ric)(X, Y) = g((\nabla_\xi Q)X+ Q(\nabla_x\xi),Y)+g(QX,\nabla_Y\xi),
  \end{equation*}
  where $X, Y\in\Gamma(M)$. Because $\xi$ is a Killing vector field, so $\mathsterling_\xi \Ric=0$, using this and \ref{pre:kContactFormula1} in the above equation give \ref{pre:kContactCovDiffXi}, and it completes the proof.
\end{proof}

An almost contact pseudo-metric structure $(\varphi, \xi, \eta, g)$ is called normal if $[\varphi, \varphi]+2d\eta\otimes\xi=0$. A normal contact pseudo-metric manifold is a Sasakian pseudo-metric manifold. A Sasakian pseudo-metric manifold is a $K-$contact pseudo-metric manifold, satisfying
\begin{gather}
  (\nabla_x\varphi)Y = g(X, Y)\xi-\epsilon\eta(Y)X,\label{pre:sasakianCri}\\
  R(X, Y)\xi = \eta(Y)X-\eta(X)Y,\label{pre:sasakianCurvature1}
\end{gather}
where $X, Y\in\Gamma(M)$.
\begin{lemma}
\label{pre:qVarphi}
  Let $(M, \varphi, \xi, \eta, g)$ be a Sasakian pseudo-metric manifold then $Q\varphi = \varphi Q$.
\end{lemma}
\begin{proof}
  First, calculating the curvature tensor by \ref{pre:sasakianCri}, we have
  \begin{equation*}
    \begin{split}
      R(X, Y, \varphi Z, W)+R(X, &Y, Z, \varphi W) =\\
      &\epsilon g(Z, \varphi Y)g(X, W)-\epsilon g(Z, \varphi X)g(Y, W)\\
      &-\epsilon g(X, Z)g(\varphi Y, W)+\epsilon g(Y, Z)g(\varphi X, W),
    \end{split}
  \end{equation*}
  where $X, Y, Z, W\in\Gamma(M)$ and $R(X, Y, Z, W)=g(R(X, Y, Z), W)$. Now, let $X, Y, Z, W$ be orthogonal to $\xi$, then using the above equation, we have
  \begin{equation*}
    R(X, Y, Z, W) = R(\varphi X, \varphi Y, \varphi Z, \varphi W),
  \end{equation*}
  and this gives
  \begin{equation*}
    \Ric(X, \varphi Y)+\Ric(\varphi X, Y)=0,
  \end{equation*}
  where $X, Y$ are orthogonal vector fields. Using the last equation, we obtain $Q\varphi = \varphi Q$, completing the proof.
\end{proof}

A contact pseudo-metric manifold $(M, \varphi, \xi, \eta, g)$ is called an $\eta-$Einstein manifold if the Ricci curvature is of the form $\Ric=a g+b\eta\otimes\eta$, where $a, b$ are smooth functions on the manifold $M$. If  the manifold $M$ is a $K-$contact pseudo-metric manifold with dimension greater than three, then $a, b$ are constants.

Let $(M, \varphi, \xi, \eta, g)$ be a contact pseudo-metric manifold, for any constant real number $t\neq0$, is defined a contact pseudo-metric manifold $(M, \tilde{\varphi}, \tilde{\xi}, \tilde{\eta}, \tilde{g})$, where $\tilde{\eta}=t\eta$, $\tilde{\xi}=\frac{1}{t}\xi$, $\tilde{\varphi}=\varphi$ and $\tilde{g}=tg+\epsilon(t-1)\eta\otimes\eta$. This transition is called a $\mathcal{D}-$homothetic deformation and it preserves some basic properties such as being $K-$contact and, in particular being Sasakian. A $\mathcal{D}-$homothetic deformation of an $\eta-$Einstein $K-$contact pseudo-metric manifold with $\Ric=ag+b\eta\otimes\eta$ is an $\eta-$Ricci $K-$contact pseudo-metric manifold such that $\tilde{\Ric}=(\frac{a-2\epsilon t+2\epsilon}{t})\tilde{g}+(2n-\tilde{a})\tilde{\eta}\otimes\tilde{\eta}$, notice that when $a=-2\epsilon$ then the Ricci tensor form is not changed. Thus, we have the following definition.
\begin{definition}
  An $\eta-$Einstein $K-$contact pseudo-metric manifold with $a=-2\epsilon$ is said to be $\mathcal{D}-$homothetically fixed.
\end{definition}

\section{$\eta-$Ricci solitons on Sasakian pseudo-metric manifolds}\label{mainResults}
In this section, we have studied $\eta-$Ricci solitons on Sasakian pseudo-metric manifolds.
\begin{theorem}
\label{thRicEq}
  Let $(M, \varphi, \xi, \eta, g)$ be a $(2n+1)$-dimensional Sasakian pseudo-metric manifold. If $(g, V, \lambda, \mu)$ be an $\eta$-Ricci soliton on the manifold $M$, then $M$ is an $\eta-$Einstein manifold and
  \begin{gather}
    \Ric = (n\epsilon+\frac{\mu\epsilon-\lambda}{2})g+(\frac{n}{2}(\epsilon+1)+\frac{\lambda}{4}(\epsilon+1)+\frac{(\epsilon-3)}{4}\mu)\eta\otimes\eta,\label{th:ricEq}\\
    r = \frac{1}{4} \epsilon  \left(\lambda -\mu +8 n^2+(4 \mu +6) n\right)+\frac{1}{4} (-\lambda +\mu -4 \lambda  n+2 n),\label{th:ricEqR1}
  \end{gather}
  where $\Ric$ and $r$ are the Ricci tensor and the scalar curvature of the metric $g$, respectively.
\end{theorem}
\begin{proof}
  Using \ref{pre:etaRicciSolitonEq} in the following formula \cite[p. 23]{yanoIntegral}
  \begin{equation*}
    (\mathsterling_V\nabla_xg-\nabla_X\mathsterling_Vg-\nabla_{[V, X]}g)(Y, Z) = -g((\mathsterling_V\nabla)(X, Y), Z)-g((\mathsterling_V\nabla)(X, Z), Y),
  \end{equation*}
  where $X, Y$ and $Z$ are arbitrary vector fields, we find
  \begin{equation}
    \begin{split}
      g((\mathsterling_V\nabla)(X, Y), Z) = &(\nabla_Z(\Ric+\mu\eta\otimes\eta))(X, Y)\\
      -&(\nabla_X(\Ric+\mu\eta\otimes\eta))(Y, Z)\\
      -&(\nabla_Y(\Ric+\mu\eta\otimes\eta))(Z, X),\qquad\forall X, Y, Z \in \Gamma(M).\label{thRicEq:1}
    \end{split}
  \end{equation}

 Using lemma \ref{pre:kContactCovDiffLemma} and lemma \ref{pre:qVarphi} we obtain $\nabla_\xi Q=0$ . Substituting $\xi$ for $Y$ in \ref{thRicEq:1}, using the foregoing equation and lemma \ref{pre:kContactCovDiffLemma} give
  \begin{equation}
    (\mathsterling_V\nabla)(X, \xi) = (4n+2\mu)\varphi X-2\epsilon Q\varphi X,\qquad \forall X \in\Gamma(M).\label{thRicEq:2}
  \end{equation}
  Differentiating \ref{thRicEq:2} along an arbitrary vector field $Y$ and using \ref{pre:sasakianCri} yield
  \begin{equation}
    \begin{split}
      (\nabla_Y\mathsterling\nabla)(X, \xi)-\epsilon(\mathsterling_V\nabla)(X, \varphi Y)& = 2\mu g(X, Y)\xi-(4n+2\mu)\epsilon\eta(X)Y\\
      &-2\epsilon(\nabla_YQ)(\varphi X)+2\eta(X)QY,\qquad\forall X, Y\in \Gamma(M).\label{thRicEq:3}
    \end{split}
  \end{equation}
  Using \ref{thRicEq:3} in the following commutative formula \cite{yanoIntegral}
  \begin{equation}
    (\mathsterling_VR)(X, Y)Z = (\nabla_X\mathsterling_V\nabla)(Y, Z)-(\nabla_Y\mathsterling_V\nabla)(X, Z),\label{thRicEq:R1}
  \end{equation}
  where $X, Y, Z$ are arbitrary vector fields, we find:
  \begin{equation}
    \begin{split}
      (\mathsterling_V R)(X, Y) & = 2\epsilon(\nabla_Y Q)\varphi X-2\epsilon(\nabla_X Q)\varphi Y\\
      &+2\eta(Y)QX-2\eta(X)QY\\
      &+(4n+2\mu)\epsilon\eta(X)Y-(4n+2\mu)\epsilon\eta(Y)X\\
      &+\epsilon(\mathsterling_V\nabla)(Y, \varphi X)-\epsilon(\mathsterling_V\nabla)(X, \varphi Y), \qquad \forall X, Y, Z\in\Gamma(M).\label{thRicEq:4}
    \end{split}
  \end{equation}
  Substituting $\xi$ for $Y$ in \ref{thRicEq:4} and using \ref{thRicEq:2}, we obtain
  \begin{equation}
    (\mathsterling_VR)(X, \xi)\xi = 4QX-4\epsilon(2n+\mu)X+4\mu\epsilon\eta(X)\xi, \qquad\forall X\in\Gamma(M).\label{thRicEq:5}
  \end{equation}
  Using \ref{pre:etaRicciSolitonEq}, we have
  \begin{equation*}
    (\mathsterling_V g)(X, \xi)+(4n+2\lambda\epsilon+2\mu)\eta(X) = 0, \qquad\forall X\in\Gamma(M),
  \end{equation*}
  and this equation yields
  \begin{gather}
    \epsilon(\mathsterling_V\eta)(X)-g(X, \mathsterling_V\xi)+2(2n+\lambda\epsilon+\mu)\eta(X)=0, \label{thRicEq:6}\\
    \eta(\mathsterling_V\xi) = (2n\epsilon+\mu\epsilon+\lambda), \label{thRicEq:7}
  \end{gather}
  where $X$ is an arbitrary vector field. Next Lie-differentiating the formula $R(X, \xi)\xi=X-\eta(X)\xi$  along the vector field $V$ and using \ref{thRicEq:6}, \ref{thRicEq:7} and \ref{pre:sasakianCurvature1}, we have
  \begin{equation}
    QX=\frac{(\epsilon-1)}{4}((\mathsterling_V\eta)X)\xi+(n\epsilon+\frac{\mu\epsilon-\lambda}{2})X+
    (n+\frac{\lambda\epsilon+(1-2\epsilon)\mu}{2})\eta(X)\xi,\label{thRicEq:8}
  \end{equation}
  where $X\in\Gamma(M)$. Now using the foregoing equation and symmetry of the Ricci tensor, we deduce
  \begin{equation*}
    \frac{\epsilon(\epsilon-1)}{4}(\mathsterling_V\eta)X\eta(Y) = \frac{\epsilon(\epsilon-1)}{4}(\mathsterling_V\eta)Y\eta(X)\qquad\forall X, Y\in\Gamma(X).
  \end{equation*}
  Using the above equation and \ref{thRicEq:7}, we find \ref{th:ricEq}, and in turn it yields \ref{th:ricEqR1}, completing the proof.
\end{proof}

Theorem \oldref{thRicEq} imposes strong condition on the potential vector field of an $\eta$-Ricci soliton on a Sasakian pseudo-metric manifold. We need the following lemma to further study.

\begin{lemma}\label{etaRicciFormula}
  Let $(M, \varphi, \xi, \eta, g)$ be a $(2n+1)$-dimensional contact pseudo-metric manifold. If $\Ric=ag+b\eta\otimes\eta$, where $a, b\in\mathbf{R}$, then
  \begin{equation}
    \Ric^{ij}\Ric_{ij}+\lambda r+\mu(\epsilon a+b)=0, \label{etaRicciFormula:1}
  \end{equation}
  where $r$ is the scalar curvature of the metric $g$.
\end{lemma}
\begin{proof}
  Using \ref{pre:etaRicciSolitonEq} in the following formula \cite{yanoIntegral}
  \begin{equation*}
    \mathsterling_V\Gamma_{ij}^h = \frac{1}{2}g^{ht}\left(\nabla_j(\mathsterling_Vg_{it})+\nabla_i(\mathsterling_Vg_{jt})-\nabla_t(\mathsterling_Vg_{ij})\right),
  \end{equation*}
  where $\Gamma_{ij}^{h}$ are the Christoffel symbols of the metric $g$, we deduce
  \begin{equation*}
    \mathsterling_V\Gamma_{ij}^h=\nabla^h(\Ric_{ij}+\mu\eta_i\eta_j)-\nabla_i(\Ric^h_j+\mu\eta_j\eta^h)-\nabla_j(\Ric^h_i+\mu\eta_i\eta^h).
  \end{equation*}
  Next using the above equation in the following equation \cite{yanoIntegral}
  \begin{equation*}
    \mathsterling_VR_{kji}^h=\nabla_k(\mathsterling_V\Gamma_{ij}^h)-\nabla_j(\mathsterling_V\Gamma_{ki}^h).
  \end{equation*}
We obtain
  \begin{equation*}
    \begin{split}
      \mathsterling_VR^h_{kji}&=\nabla_k\nabla^h(\Ric_{ij}+\mu\eta_i\eta_j)-\nabla_k\nabla_i(\Ric^h_j+\mu\eta_j\eta^h)\\
      &-\nabla_k\nabla_j(\Ric^h_i+\mu\eta_i\eta^h)-\nabla_j\nabla^h(\Ric_{ki}+\mu\eta_k\eta_i)\\
      &+\nabla_j\nabla_k(\Ric_i^h+\mu\eta_i\eta^h)+\nabla_j\nabla_i(\Ric_k^h+\mu\eta_k\eta^h).
    \end{split}
  \end{equation*}
  The foregoing equation and the lemma's assumption yield
  \begin{equation*}
    \mathsterling_V\Ric_{ji}=\nabla_h\nabla^h(\Ric_{ij}+\mu\eta_i\eta_j)-\nabla_h\nabla_i(\Ric^h_i+\mu\eta_j\eta^h)-\nabla_h\nabla_j(\Ric_i^h+\mu\eta_i\eta^h).
  \end{equation*}
  Eq.\ref{pre:etaRicciSolitonEq} gives $\mathsterling_Vg^{ij}=2\Ric^{ij}+2\lambda g^{ij}+2\mu\eta^i\eta^j$, using this, \ref{pre:etaRicciSolitonEq} and the above equation we obtain \ref{etaRicciFormula:1}, and it completes the proof.
\end{proof}

\begin{theorem}
  Let $(M, \varphi, \xi, \eta, g)$ be a Sasakian pseudo-metric manifold and let $(g, V, \lambda, \mu)$ be an $\eta-$Ricci soliton on $M$.
  \begin{enumerate}[(a)]
  \item If $\xi$ is a timelike vector field then, $V$ is a Killing vector field.
  \item If $\xi$ is a spacelike vector field and $V$ is not a Killing vector field, then $M$ is $\mathcal{D}-$homothetically fixed and $\mathsterling_V\varphi=0$.
  \end{enumerate}
\end{theorem}
\begin{proof}
  In the case of (a), using \ref{revBasic:pro1}, we find $\lambda-\mu=2n$, this, \ref{th:ricEq} and \ref{pre:etaRicciSolitonEq} yield $V$ is Killing.

  In the case of (b), using lemma \ref{etaRicciFormula}, we obtain $(-\lambda +\mu +2 n+4) (\lambda +\mu +2 n)=0$. According to the theorem's assumption $V$ is not a Killing vector field, so $\lambda -\mu  = 2 n+4$, using this in \ref{th:ricEq}, we deduce
  \begin{equation}
    \Ric=-2g+2(n+1)\eta\otimes\eta,\label{thPotentialVec:1}
  \end{equation}
  so $M$ is $\mathcal{D}-$homothetically fixed. Using the foregoing equation and \ref{thRicEq:1}, we obtain
  \begin{equation*}
      (\mathsterling_V\nabla)(Y, Z)=2(2n+2+\mu)(\eta(Z)\varphi Y+\eta(Y)\varphi Z), \qquad\forall Y, Z\in\Gamma(M).
    \end{equation*}
    Differentiating the above equation along an arbitrary vector field $X$, using \ref{thRicEq:R1}, and contracting at $X$, we have
    \begin{equation}
      (\mathsterling_V\Ric)(Y, Z) = 2(2n+2+\mu)(2g(Y, Z)-(4n+2)\eta(Y)\eta(Z)),\label{thPotentialVec:2}
    \end{equation}
    where $Y, Z$ are arbitrary vector fields.

    Next using \ref{thPotentialVec:1} in \ref{pre:etaRicciSolitonEq}, we find
    \begin{equation}
      (\mathsterling_Vg)(Y, Z)=-(2n+\lambda+\mu)(g+\eta\otimes\eta)(Y, Z), \qquad\forall Y, Z\in\Gamma(M).\label{thPotentialVec:R1}
    \end{equation}
    Lie-differentiating \ref{thPotentialVec:1} along the vector field $V$ gives us
    \begin{equation}
      \begin{split}
        (\mathsterling_V\Ric)(Y, Z) = &2(2n+\lambda+\mu)\{g(Y, Z)+\eta(Y)\eta(Z)\}+\\
        &2(n+1)\{\eta(Z)(\mathsterling_V\eta)Y+\eta(Y)(\mathsterling_V\eta)Z\},\qquad\forall Z, Y\in\Gamma(M).\label{thPotentialVec:3}
      \end{split}
    \end{equation}
    Substituting $\xi$ for $Y$ in \ref{thPotentialVec:3} and \ref{thPotentialVec:2}, using \ref{thRicEq:7}, we obtain
    \begin{equation}
      (\mathsterling_V\eta)Y=-2(2+2n+\mu)\eta(Y), \qquad\forall Y\in\Gamma(M).\label{thPotentialVec:4}
    \end{equation}
    Operating the above equation by $d$ and noticing the fact that $d$ commutes with Lie-derivative we deduce
    \begin{equation*}
      (\mathsterling_Vd\eta)(X, Y) = -2(2+2n+\mu)g(X, \varphi Y), \qquad\forall X, Y\in \Gamma(M).
    \end{equation*}
    Lie-differentiating  \ref{pre:basicPro1} along the vector field $V$ and using the above equation, yield $\mathsterling_V\varphi=0$, and it completes the proof.
  \end{proof}
  \begin{example}
    Consider $\mathbf{R}^3$ with the standard coordinate system $(x, y, z)$. Let $\xi = 2\frac{\partial}{\partial z}$, $\eta=\frac{1}{2}(-ydx+dz)$, $\varphi (\frac{\partial}{\partial x}) = -\frac{\partial}{\partial y}$, $\varphi(\frac{\partial}{\partial y}) = \frac{\partial}{\partial x}+y\frac{\partial}{\partial z}$ and $\varphi (\frac{\partial}{\partial z})=0$. If $g=\epsilon\eta\otimes\eta+\frac{1}{4}(dx^2+dy^2)$, then $(M, \varphi, \xi, \eta, g)$ is a Sasakian pseudo-metric manifold. By direct calculation, we have $\Ric = -2\epsilon g+4\eta\otimes\eta$. Now, let $V$ be a vector field defined by
    \begin{equation*}
      V = ((2-6\epsilon+(\epsilon-1)\lambda+(1-2\epsilon)\mu)x\frac{\partial}{\partial x} + (2\epsilon-\lambda)y\frac{\partial}{\partial y}-(2+\epsilon\lambda+\mu)z\frac{\partial}{\partial z}.
    \end{equation*}
    If $\xi$ be a spacelike vector field and $\lambda-\mu=6$ then $(g, V, \lambda, \mu)$ is an $\eta-$Ricci soliton on $M$, $\mathsterling_V\varphi=0$ and $V$ is not a Killing vector field. But if $\xi$ is a timelike vector field then $(g, V, \lambda, \mu)$ is an $\eta-$Ricci soliton on $M$ iff $V$ is a Killing vector field, and this condition is satisfied if $\lambda=-2$ and $\mu=-4$.
  \end{example}
  \begin{proposition}
    Let $(M, \varphi, \xi, \eta, g)$ be a $K$-contact pseudo-metric manifold. If $(g, V, \lambda, \mu)$ is a gradient $\eta$-Ricci soliton on $M$ then $M$ is an $\eta-$Einstein manifold and $\Ric=-\lambda g-\mu\eta\otimes\eta$, where $-\epsilon\lambda-\mu=2n.$
  \end{proposition}
  \begin{proof}
    First  \ref{pre:gradientEtaRicciSolitonEq} gives
    \begin{equation}
      \nabla_XDf+QX+\lambda X+\epsilon\mu\eta(X)\xi = 0,\qquad\forall X\in\Gamma(M).\label{preposition1Rev2:formula1}
    \end{equation}
    Calculating $R(X, Y)Df$ by the above equation, we deduce
    \begin{equation}
      \begin{split}
        R(X, Y)Df&=\epsilon\mu(\nabla_Y\eta)X\xi+\epsilon\mu\eta(X)\nabla_Y\xi+(\nabla_YQ)X\\
        &-\epsilon\mu(\nabla_X\eta)Y\xi-\epsilon\mu\eta(Y)\nabla_X\xi-(\nabla_XQ)Y,\qquad\forall X, Y\in\Gamma(M).\label{rev:r1:eq1}
      \end{split}
    \end{equation}
    Substituting $\xi$ for $Y$ in the last equation and using lemma \oldref{pre:kContactCovDiffLemma}, we find
    \begin{equation*}
      R(X, \xi)Df= (\mu+2n)\varphi X-\epsilon\varphi QX,\qquad X\in\Gamma(M).
    \end{equation*}
    Scalar product of the above equation with $\xi$ gives $df=(\xi(f))\eta$, operating $d$ on this equation, we obtain $d\eta\wedge(\xi(f))+\eta\wedge d(\xi(f))=0$, taking exterior product of the last equation with $\eta$ and using $\eta\wedge d\eta\neq0$, we have $\xi(f)=0$, so $f$ is a constant function. Next using this consequence in \ref{preposition1Rev2:formula1}, we find $\Ric=-\lambda g-\mu\eta\otimes\eta$ and this gives $-\epsilon\lambda-\mu =2n$, completing the proof.
  \end{proof}

  One may ask, what will happen if the potential vector filed of an $\eta-$Ricci soliton on a contact pseudo-metric manifold $(M, \varphi, \xi, \eta, g)$ is $\xi$, we have answered this question in the following theorem.
  \begin{theorem}
    Let $(M, \varphi, \xi, \eta, g)$ be a contact pseudo-metric manifold, and let $(g, \varphi, \lambda, \mu)$ be an $\eta-$Ricci soliton on the manifold $M$. If $V$ is colinear with $\xi$ and $Q\varphi=\varphi Q$, then $M$ is an $\eta-$Einstein $K-$contact pseudo-metric manifold and $\Ric = -\lambda g-\mu\eta\otimes\eta$, where $-\epsilon\lambda-\mu=2n.$
  \end{theorem}
  \begin{proof}
    Let $V=f\xi$, where $f$ is a non-zero smooth function on the manifold $M$. Using this in \ref{pre:etaRicciSolitonEq}, we have
    \begin{equation}
      \epsilon X(f)\eta(Y)+\epsilon Y(F)\eta(X)-2fg(\varphi hX, Y)+2\Ric(X, Y)+2\lambda g(X, Y)+2\mu\eta(X)\eta(Y)=0,\label{th2:eq1}
    \end{equation}
    for $X, Y\in\Gamma(M)$. Substituting $\xi$ for $Y$ in \ref{th2:eq1}, we deduce
    \begin{equation}
      \epsilon Df+2Q\xi+(\xi(f)+2\lambda+2\epsilon\mu)\xi=0.\label{th2:eq2}
    \end{equation}
    By assumption $Q\varphi=\varphi Q$, so $\varphi Q\xi=0$, using this and \ref{pre:basicPro2}, we have $Q\xi=\epsilon(2n-trh^2)\xi$. Substituting this consequence in \ref{th2:eq2}, we find
    \begin{equation}
      \epsilon Df+(2\epsilon(2n-trh^2)+\xi(f)+2\lambda+2\mu\epsilon)\xi=0.\label{th2:eq3}
    \end{equation}
    Next, substituting $\xi$ for $X, Y$ in \ref{th2:eq1}, we obtain
    \begin{equation*}
      2n-trh^2=-\epsilon(\xi(f))-\lambda\epsilon-\mu.
    \end{equation*}
    The above equation and \ref{th2:eq3} give $Df=\epsilon(\xi(f))\xi$, differentiating this equation along an arbitrary vector field $X$ and using \ref{pre:basicPro3}, we find
    \begin{equation*}
      g(\nabla_X(Df), Y)=X(\xi(f))\eta(Y)-\epsilon\xi(f)\{g(\epsilon\varphi X, Y)+g(\varphi hX, Y)\},\quad\forall X, Y\in \Gamma(M).
    \end{equation*}
    Using the above equation, \ref{pre:basicPro1} and the known formula  $g(\nabla_X(Df), Y)=g(\nabla_Y(Df), X)$, where $X, Y\in\Gamma(M)$, we deduce
    \begin{equation*}
      X(\xi(f))\eta(Y)-Y(\xi(f))\eta(X)=-2\xi(f)d\eta(X, Y),\qquad\forall X, Y\in\Gamma(M).
    \end{equation*}
    Considering $X,Y$ as arbitrary orthogonal vector fields to $\xi$ in the above equation and noticing that $d\eta\neq0$, we deduce $X(f)=0$, so $f$ is a constant function on the manifold $M$. Using this consequence in \ref{th2:eq1} gives
    \begin{equation}
      -f\varphi hX+QX+\lambda X+\epsilon\mu\eta(X)\xi=0,\qquad\forall X\in\Gamma(M).\label{th2:eq4}
    \end{equation}
    Substituting $\varphi X$ for $X$ in the above equation, we find
    \begin{equation}
      -f\varphi h\varphi X+QX+\lambda X=0,\qquad\forall X\in\Gamma(M).\label{th2:eq5}
    \end{equation}
    Operating $\varphi$ on \ref{th2:eq4} and using $\varphi h=-h\varphi$, we have
    \begin{equation}
      f\varphi h\varphi X+QX+\lambda X=0,\qquad\forall X\in\Gamma(M).\label{th2:eq6}
    \end{equation}
    Using the above equation, \ref{th2:eq5}, \ref{pre:basicPro4} and $Q\xi=(-\lambda-\mu\epsilon)\xi$, we obtain:
    \begin{equation*}
      \Ric=-\mu\eta\otimes\eta-\lambda g.
    \end{equation*}
    Using the above equation in \ref{pre:etaRicciSolitonEq} gives $\mathsterling_\xi g=0$, so $M$ is a $K-$contact pseudo-metric manifold and $-\epsilon\lambda-\mu=2n$, completing the proof.
  \end{proof}
  \section{$\eta-$Ricci solitons on $(\kappa, \mu)-$contact pseudo-metric manifolds}\label{kappaMuSection}
  Studying nullity conditions on manifolds is one of the interesting topics in differential geometry, specially in the context of contact pseudo-metric manifolds.   In \cite{ghaffarzadeh}, Ghaffarzadeh and second author introduced the notion of a $(\kappa, \mu)-$contact pseudo-metric manifold. According to them a contact pseudo-metric manifold $(M, \varphi, \xi, \eta)$ is called a $(\kappa, \mu)-$contact pseudo-metric manifold if it satisfies
  \begin{equation}
    R(X, Y)\xi = \epsilon\kappa(\eta(Y)X-\eta(X)Y)+\epsilon\mu(\eta(Y)hX-\eta(X)hY),\label{km:rev01:1}
  \end{equation}
  where $R$ is the Riemannian curvature tensor of $M$, $\kappa, \mu$ are constant real numbers, and $X, Y$ are arbitrary vector fields. For a $(\kappa, \mu)-contact$ pseudo-metric manifold we have the following formulas \cite{ghaffarzadeh}
  \begin{gather}
    h^2=(\epsilon\kappa-1)\varphi^2,\label{km:basic1}\\
    Q\xi=2n\kappa\xi,\label{km:basic2}\\
    (\nabla_\xi h)=-\epsilon\mu\varphi h,\label{km:basic3}
  \end{gather}
  furthermore if  $\epsilon\kappa<1$ then we have \cite{ghaffarzadeh}
  \begin{gather}
    \begin{split}
      QX&=\epsilon[2(n-1)-n\mu]X+(2(n-1)+\mu)hX\\
      &+[2(1-n)\epsilon+2n\kappa+n\epsilon\mu]\eta(X)\xi, \label{km:condRicciTen}
    \end{split}\\
    r = 2n(\kappa-2\epsilon)+2n^2\epsilon(2-\mu),\label{km:condScalCurv}
  \end{gather}
  where $X$ and $r$ are, an arbitrary vector field and the scalar curvature of the manifold, respectively.

  \begin{lemma}
    Let $(M, \varphi, \xi, \eta, g)$ be a $(\kappa, \mu)-$contact pseudo-metric manifold, and let $\epsilon\kappa<1$. If $(g, V, \lambda, \tau)$ is a gradient $\eta-$Ricci soliton on the manifold $M$ then
    \begin{equation}
      \epsilon\kappa(-2+\mu)=n\mu+\mu+\tau\label{kmLemma1:1}.
    \end{equation}
  \end{lemma}
  \begin{proof}
    Differentiating \ref{km:basic2} along an arbitrary vector field $X$ and using \ref{pre:basicPro3}, we deduce
    \begin{equation}
      (\nabla_X Q)\xi=Q(\epsilon\varphi+\varphi h)X-2n\kappa(\epsilon\varphi+\varphi h)X,\qquad\forall X\in\Gamma(M).\label{kmlm1:eq1}
    \end{equation}
    Taking scalar product of \ref{rev:r1:eq1} and $\xi$, and using \ref{kmlm1:eq1}, we have
    \begin{equation}
      \begin{split}
        g(R(X, Y)Df, \xi)&=\epsilon g((Q\varphi+\varphi Q)Y, X)+g((Q\varphi h+h\varphi Q)Y, X)\\
        &+(-4n\kappa\epsilon-2\tau)g(\varphi Y, X),\quad\forall X, Y\in\Gamma(M).\label{kmth:rev:01}
      \end{split}
    \end{equation}
    Substituting $\varphi X$ for $X$ and $\varphi Y$ for $Y$ in \ref{km:rev01:1} give $R(\varphi X, \varphi Y)\xi=0$, using this, \ref{pre:basicPro4} and  the above equation, we obtain
    \begin{equation}
      \epsilon(\varphi Q+Q\varphi)X-(\varphi Qh+hQ\varphi)X+(-4n\kappa\epsilon-2\tau)\varphi X=0,\label{kmlm1:eq2}
    \end{equation}
    where $X$ is an arbitrary vector field. Now, substituting $\varphi X$ for $X$ in \ref{km:condRicciTen}, we have
    \begin{equation}
      Q\varphi X=\epsilon[2(n-1)-n\mu]\varphi X+(2(n-1)+\mu)h\varphi X,\qquad \forall X\in\Gamma(M).\label{kmlm:eq3}
    \end{equation}
    Next, operating $\varphi$ on \ref{km:condRicciTen}, we obtain
    \begin{equation}
      \varphi Qx=\epsilon[2(n-1)-n\mu]\varphi X+(2(n-1)+\mu)\varphi hX,\qquad\forall X\in\Gamma(M).\label{kmlm:eq4}
    \end{equation}
    Substituting $hX$ for $X$ in \ref{kmlm:eq4}, using \ref{km:basic1} and \ref{pre:basicPro4} give
    \begin{equation}
      \varphi QhX=\epsilon[2(n-1)-n\mu]\varphi hX-(\epsilon\kappa-1)(2(n-1)+\mu)\varphi X,\qquad X\in\Gamma(M).
    \end{equation}
    Operating $h$ on \ref{kmlm:eq3}, using \ref{km:basic1} and \ref{pre:basicPro4}, we have
    \begin{equation}
      hQ\varphi X=\epsilon[2(n-1)-n\mu]h\varphi X-(\epsilon\kappa-1)(2(n-1)+\mu)\varphi X,\qquad X\in\Gamma(M).
    \end{equation}
    Using the last four equations in \ref{kmlm1:eq2} and  $h\varphi=-\varphi h$ give \ref{kmLemma1:1}, completing the proof.
  \end{proof}
  \begin{theorem}\label{kmManth1}
    Let $(M, \varphi, \xi, \eta, g)$ be a $(\kappa, \mu)-$contact pseudo-metric manifold and let $\epsilon\kappa<1$. If $(g, V, \lambda, \tau)$ is a gradient $\eta-$Ricci soliton on $M$, then $\mu=0, \tau=-2\epsilon\kappa$, or $\Ric=-\lambda g-\tau\eta\otimes\eta$ and $\mu=2-2n, \tau=2n(-\frac{1}{n}+n-\epsilon\kappa)$.
  \end{theorem}
  \begin{proof}
    First, substituting $\xi$ for $X$ in \ref{kmth:rev:01}, using \ref{km:rev01:1} and \ref{km:basic2}, we have
    \begin{equation*}
      \kappa(\xi(f)\xi-\epsilon Df)-\epsilon\mu hDf=0.
    \end{equation*}
    Differentiating the above equation along vector field $\xi$ and using \ref{km:basic3}, we have
    \begin{equation*}
      \kappa\xi(\xi(f))\xi+\epsilon\kappa(2n\kappa+\lambda+\tau\epsilon)\xi+(\epsilon\mu)^2\varphi h Df=0.
    \end{equation*}
    Now, operating $\varphi$ on the last equation we find $\mu^2h Df=0$, taking $h$ from this and using \ref{km:basic1}, we obtain
    \begin{equation*}
      \mu^2(\epsilon\kappa-1)(-Df+\eta(Df)\xi)=0.
    \end{equation*}
    Examining the above equation we have either, i)$\mu=0$ or ii)$\mu\neq0$.

    In the case i), using \ref{kmLemma1:1}, we obtain $\tau=-2\epsilon\kappa$.
    In the case ii), we have $Df=\eta(Df)\xi$, differentiating this along arbitrary vector field $X$ and using \ref{pre:basicPro3}, we have
    \begin{equation*}
      g(\nabla_X Df, Y) = X(\xi(f))\eta(Y)-\xi(f)g(\varphi X, Y)-\epsilon\xi(f)f(\varphi hX, Y),
    \end{equation*}
    where $X, Y$ are arbitrary vector fields. Using the above equation and $g(\nabla_X Df, Y)=g(\nabla_Y Df, X)$, we find
    \begin{equation*}
      X(\xi(f))\eta(Y)-Y(\xi(f))\eta(X)+2\xi(f)d\eta(X, Y)=0,\qquad\forall X, Y\in\Gamma(M).
    \end{equation*}
    Substituting $\varphi X$ for $X$ and $\varphi Y$ for $Y$, and noticing the fact that $d\eta\neq0$, we have $\xi(f)=0$. So $f$ is a constant function and $\Ric=-\lambda g-\tau\eta\otimes\eta$. Using this gives $r=(2n+1)(-\lambda)-\epsilon\tau$, comparing the last consequence and \ref{km:condScalCurv}, we have
    \begin{equation*}
      n\mu=-2+2n-2n\kappa\epsilon-\tau.
    \end{equation*}
    Now using the above equation and \ref{kmLemma1:1} we obtain, $\mu=2-2n$ and $\tau=2n(-\frac{1}{n}+n-\epsilon\kappa)$, completing the proof.
  \end{proof}
  \begin{corollary}
    Let $(M, \varphi, \xi, \eta, g)$ be a $(\kappa, \mu)-$contact pseudo-metric manifold and let $\epsilon\kappa<1$. If $(g, V, \lambda, 0)$ is a gradient $\eta-$Ricci soliton (in fact a gradient Ricci soliton) on $M$, then $R(X, Y)\xi=0$, where $X, Y$ are arbitrary vector fields.
  \end{corollary}

\end{document}